\documentclass[10pt]{article}

\usepackage[margin=1.05in]{geometry}
\usepackage{amsmath,amssymb,amsthm,mathtools}
\usepackage{enumitem}
\usepackage[colorlinks=true,citecolor=blue,linkcolor=blue,urlcolor=blue]{hyperref}
\usepackage{xcolor}
\usepackage{microtype}
\usepackage[section]{placeins}
\usepackage{tikz}
\usetikzlibrary{arrows.meta,positioning,fit,calc,backgrounds}

\title{A Nonabelian Twist on Differences of Bijections}
\author{
Mohsen Aliabadi\\
Department of Mathematics, Clayton State University\\
Morrow, GA, USA\\
\texttt{MohsenAliabadi@clayton.edu, mohsenmath88@gmail.com}
}

\date{}

\newtheorem{theorem}{Theorem}[section]
\newtheorem{lemma}[theorem]{Lemma}
\newtheorem{proposition}[theorem]{Proposition}
\newtheorem{corollary}[theorem]{Corollary}
\newtheorem{definition}[theorem]{Definition}

\newtheorem{problem}[theorem]{Problem}
\newtheorem{remark}[theorem]{Remark}

\DeclareMathOperator{\Sym}{Sym}

\begin{document}

\maketitle

\begin{abstract}
Hall's theorem on differences of bijections characterizes the multisets
\(\{a_1,\ldots,a_{|G|}\}\) in a finite abelian group \(G\) that can be written
in the form
\[
        a_i=b_i-c_i,
\]
where both \(b_1,\ldots,b_{|G|}\) and \(c_1,\ldots,c_{|G|}\) are enumerations
of \(G\). The necessary and sufficient condition is the zero-sum condition
\[
        a_1+\cdots+a_{|G|}=0.
\]
This paper studies the corresponding problem for finite nonabelian groups,
with differences replaced by quotients. Thus we ask when a multiset \(A\) of
cardinality \(|G|\) can be represented as
\[
        A=\{b(i)c(i)^{-1}:1\le i\le |G|\},
\]
where \(b\) and \(c\) are bijections onto \(G\).

Passing to the abelianization gives a necessary condition, namely that the
product of the images of the elements of \(A\) is trivial in \(G_{\rm ab}\).
We show that this condition is not sufficient in general, even when the
elements of \(A\) admit an ordering whose product is the identity in \(G\).
The main structural result is a cycle-tiling criterion: quotient-realizability
is equivalent to a decomposition of \(A\) into product-one words whose
partial-product sets tile \(G\) by right translates. The use of permutation
cycles is standard, but the criterion translates quotient-realizability into
an exact tiling condition. We then use this criterion to construct a
counterexample in \(S_3\), and we extend the same obstruction to infinitely
many finite nonabelian groups.
\end{abstract}



\section{Introduction}\label{sec:introduction}

Hall's theorem on differences of bijections is a classical result in
combinatorial group theory. Let \(G\) be a finite abelian group, written
additively, and let
\[
        a_1,\ldots,a_{|G|}\in G.
\]
One asks when there exist two enumerations
\[
        b_1,\ldots,b_{|G|} \qquad\text{and}\qquad
        c_1,\ldots,c_{|G|}
\]
of the elements of \(G\) such that
\[
        a_i=b_i-c_i
        \qquad (1\le i\le |G|).
\]
There is an immediate necessary condition. Summing the equations gives
\[
        a_1+\cdots+a_{|G|}=0,
\]
because the \(b_i\)'s and the \(c_i\)'s are two enumerations of the same finite
abelian group. Hall proved that this condition is also sufficient
\cite{Hall1952}. Thus, for finite abelian groups, the zero-sum obstruction is
the only obstruction.

The problem has several equivalent or related formulations. In the cyclic
group \(\mathbb Z/n\mathbb Z\), it may be viewed as a problem of placing
directed arrows of prescribed lengths on a regular \(n\)-gon so that no two
arrows have the same tail and no two arrows have the same head. Ullman and
Velleman \cite{UllmanVelleman2019} give a detailed exposition of this
viewpoint and explain connections with juggling sequences, bus scheduling,
infinite abelian groups, and partial transversals in Latin squares. The finite
abelian case is also closely related to the earlier work of Salzborn and
Szekeres \cite{SalzbornSzekeres1979}, while the infinite abelian case goes
back to Fuchs \cite{Fuchs1958}.

The purpose of the present paper is to examine what remains true when the
group is not assumed to be abelian. We write groups multiplicatively. For a
finite possibly nonabelian group \(G\), the natural analogue of a difference
is a quotient. Thus we ask whether a multiset \(A\) of cardinality \(|G|\),
with terms in \(G\), can be written as
\[
        A=\{b(i)c(i)^{-1}:1\le i\le |G|\},
\]
where
\[
        b,c:\{1,\ldots,|G|\}\longrightarrow G
\]
are bijections. We call such a multiset \emph{quotient-realizable}; see
Definition \ref{def:quotient-realizable}.

If \(G\) is abelian, then this is exactly Hall's theorem written
multiplicatively. Indeed, a multiset
\[
        A=\{a_1,\ldots,a_{|G|}\}
\]
is quotient-realizable in an abelian group if and only if
\[
        a_1a_2\cdots a_{|G|}=1.
\]
For a nonabelian group there is still a necessary condition obtained by
passing to the abelianization. If
\[
        \pi:G\longrightarrow G_{\rm ab}=G/[G,G]
\]
is the canonical projection and \(A\) is quotient-realizable, then
\[
        \prod_{a\in A}^{\rm mult}\pi(a)=1
        \qquad\text{in }G_{\rm ab}.
\]
This product is well-defined because \(G_{\rm ab}\) is abelian. We refer to
this as the \emph{abelianization obstruction}; see
Proposition \ref{prop:abelianization-obstruction}.

One might hope that the abelianization obstruction is sufficient, in direct
analogy with Hall's theorem. The main point of the present paper is that this
is not the case. We show that even the stronger condition that the elements of
\(A\) can be ordered so that their product is \(1\) in \(G\) does not imply
quotient-realizability. The additional obstruction is not detected by a single
product. It comes from the cycle structure of a permutation of \(G\).

Indeed, a quotient representation is equivalent to a permutation
\(\varphi\in\Sym(G)\) satisfying
\[
        A=\{\varphi(x)x^{-1}:x\in G\};
\]
see Lemma \ref{lem:single-permutation}. Thus the directed graph whose edges are
\[
        x\longrightarrow \varphi(x)
\]
is a disjoint union of directed cycles. The label on the edge
\(x\to \varphi(x)\) is \(\varphi(x)x^{-1}\). Along each cycle, the labels form
a product-one word. More precisely, if a cycle is
\[
        x_0\mapsto x_1\mapsto \cdots \mapsto x_{\ell-1}\mapsto x_0,
\]
then the corresponding labels
\[
        g_j=x_jx_{j-1}^{-1}\quad (1\le j\le \ell-1),
        \qquad
        g_\ell=x_0x_{\ell-1}^{-1}
\]
satisfy
\[
        g_\ell g_{\ell-1}\cdots g_1=1.
\]
Furthermore, the partial products
\[
        p_0=1,\qquad p_j=g_jg_{j-1}\cdots g_1
\]
recover the vertices of the cycle by
\[
        p_j=x_jx_0^{-1}.
\]
Thus a quotient realization is not merely an ordering of \(A\) with product
\(1\). It is a decomposition of \(A\) into product-one words whose
partial-product sets can be placed disjointly inside \(G\) by right
translation.

This leads to the main structural result of the paper. A word
\[
        w=(g_1,\ldots,g_\ell)
\]
is called a \emph{simple product-one word} if
\[
        p_\ell=1
\]
and
\[
        p_0,p_1,\ldots,p_{\ell-1}
\]
are pairwise distinct, where \(p_j=g_jg_{j-1}\cdots g_1\). We write
\[
        P(w)=\{p_0,p_1,\ldots,p_{\ell-1}\}.
\]
Theorem \ref{thm:cycle-tiling} states that \(A\) is quotient-realizable in
\(G\) if and only if \(A\) can be partitioned, as a multiset, into simple
product-one words
\[
        w_1,\ldots,w_r
\]
such that there exist elements \(x_1,\ldots,x_r\in G\) with
\[
        G=\bigsqcup_{j=1}^r P(w_j)x_j.
\]

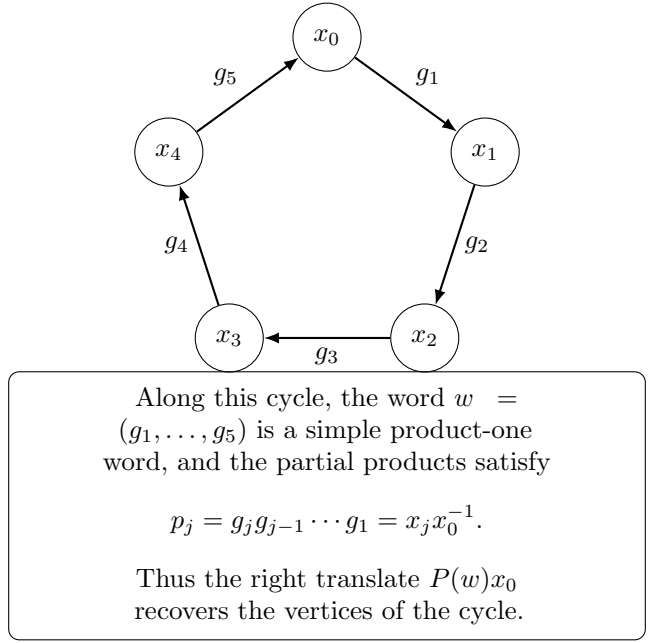
\begin{figure}[!htbp]
\centering
\begin{tikzpicture}[
    scale=1,
    vertex/.style={circle,draw,minimum size=9mm,inner sep=1pt},
    edge/.style={-{Latex[length=2mm]},thick},
    note/.style={draw,rounded corners,align=center,text width=8cm,inner sep=6pt}
]

\node[vertex] (x0) at (90:2.2) {$x_0$};
\node[vertex] (x1) at (18:2.2) {$x_1$};
\node[vertex] (x2) at (-54:2.2) {$x_2$};
\node[vertex] (x3) at (-126:2.2) {$x_3$};
\node[vertex] (x4) at (162:2.2) {$x_4$};

\draw[edge] (x0) -- node[above right] {$g_1$} (x1);
\draw[edge] (x1) -- node[right] {$g_2$} (x2);
\draw[edge] (x2) -- node[below] {$g_3$} (x3);
\draw[edge] (x3) -- node[left] {$g_4$} (x4);
\draw[edge] (x4) -- node[above left] {$g_5$} (x0);

\node[note] at (0,-4.0) {
Along this cycle, the word \(w=(g_1,\ldots,g_5)\) is a simple
product-one word, and the partial products satisfy
\[
p_j=g_jg_{j-1}\cdots g_1=x_jx_0^{-1}.
\]
Thus the right translate \(P(w)x_0\) recovers the vertices of the cycle.
};

\end{tikzpicture}
\caption{A directed cycle arising from a quotient realization. The edge labels form a simple product-one word \(w=(g_1,\ldots,g_5)\). Its partial-product set \(P(w)=\{p_0,p_1,\ldots,p_4\}\), where \(p_0=1\) and \(p_j=g_jg_{j-1}\cdots g_1\), recovers the vertex set after right translation by \(x_0\).}
\label{fig:cycle-word}
\end{figure}
Figure~\ref{fig:cycle-word} illustrates this mechanism. The labels along a directed cycle form a simple product-one word, and the corresponding partial-product set recovers the vertices of the cycle by right translation.
This is the basic local picture behind the cycle-tiling decompositions used below.

The cycle decomposition of a permutation is, of course, a standard tool. The
contribution here is the formulation of quotient-realizability as an exact
tiling condition by partial-product sets. This point is important for
distinguishing the present problem from several related theories. Complete
mappings and orthomorphisms also concern permutations of a finite group whose
associated difference or quotient maps are bijective; see, for example, the
work of Hall and Paige \cite{HallPaige1955} and the survey of Evans
\cite{Evans2015}. Those problems correspond to the special case in which the
quotient map itself is a bijection. By contrast, in the present paper the
quotient multiset is prescribed in advance and may have repeated elements.
Thus the problem is not the existence of one complete mapping, but rather the
realizability of an arbitrary multiset of quotients.

There are also connections with Latin squares and transversals. Complete
mappings of a finite group correspond to transversals in the Cayley table, and
orthomorphisms correspond to orthogonality constructions for Latin squares.
Standard references for this viewpoint include D\'enes and Keedwell
\cite{DenesKeedwell1974}, Evans \cite{Evans2015}, and Wanless
\cite{Wanless2011}. Ullman and Velleman \cite{UllmanVelleman2019} also
explain how differences of bijections may be interpreted in terms of partial
transversals. The present quotient-realizability problem can be viewed in this
same general setting, but with a prescribed multiset of symbols rather than a
single transversal using each symbol once.

The partial-product aspect of our criterion is related in spirit to group
sequencings, terraces, and \(R\)-sequencings. These topics study orderings of
group elements whose partial products or quotients have prescribed
distinctness properties; see the survey of Ollis \cite{Ollis2012}. The
difference is that our criterion allows several cycles rather than one global
ordering, and it requires the corresponding partial-product sets to tile the
ambient group by right translates. In this sense, the obstruction obtained
here is a tiling obstruction associated with the cycle decomposition of a
permutation. Cycle decompositions of Cayley graphs form another related
background; see, for instance, the survey of Alspach, Bermond, and Sotteau
\cite{AlspachBermondSotteau1990}.

The paper also intersects with zero-sum and product-one theory. In abelian
groups, Hall's condition is a zero-sum condition. In nonabelian groups,
unordered sequences whose terms can be ordered to have product \(1\) are often
called product-one sequences. Such sequences and their associated Davenport
constants have been studied extensively; see, for example, Gao and Geroldinger
\cite{GaoGeroldinger2006} for the abelian theory and Geroldinger and
Grynkiewicz \cite{GeroldingerGrynkiewicz2014} for product-one sequences over
nonabelian groups. The examples in this paper show that product-one
orderability alone is still not enough for quotient-realizability. The cycle
tiles must also fit inside the group.

We use the cycle-tiling criterion in two main ways. First, we prove a
subgroup-support theorem. If \(H\le G\) and the multiset \(A\) is supported in
\(H\), then every quotient realization preserves the right cosets of \(H\).
Consequently, \(A\) is quotient-realizable in \(G\) if and only if \(A\) can
be partitioned into \([G:H]\) submultisets of size \(|H|\), each
quotient-realizable inside \(H\); see Theorem \ref{thm:subgroup-support}.
When \(H\) is abelian, this becomes a zero-sum block partition condition; see
Corollary \ref{cor:abelian-subgroup}.

Second, we construct an explicit obstruction in \(S_3\). Let
\[
        G=S_3,\qquad s=(12),\qquad t=(23),
\]
and consider the multiset
\[
        A=\{s,s,t,t,t,t\}.
\]
This multiset satisfies the abelianization obstruction, since it contains an
even number of transpositions. It also admits an ordering with product \(1\),
for instance
\[
        ss\,tt\,tt=1.
\]
Nevertheless, \(A\) is not quotient-realizable in \(S_3\). The reason is that
a simple product-one word using only \(s\) and \(t\) must be either
\[
        (s,s),\qquad (t,t),
\]
or an alternating word of length \(6\) containing three copies of \(s\) and
three copies of \(t\). Since \(A\) contains two copies of \(s\) and four copies
of \(t\), any cycle-tiling decomposition would have to tile \(S_3\) by one
right coset of \(\langle s\rangle\) and two right cosets of
\(\langle t\rangle\). Such a tiling is impossible; see
Proposition \ref{prop:S3-counterexample}.

Finally, we show that this phenomenon is not isolated. If \(K\) is any finite
group with \(3\nmid |K|\), then
\[
        S_3\times K
\]
admits a multiset satisfying the abelianization obstruction and admitting a
product-one ordering, but not quotient-realizable; see
Theorem \ref{thm:infinite-family}. Thus the abelianization condition is not a
sufficient criterion in the class of finite nonabelian groups.

The paper is organized as follows. Section \ref{sec:quotient-realizable}
introduces quotient-realizable multisets and records the abelianization
obstruction. Section \ref{sec:cycle-tiling} proves the cycle-tiling criterion.
Section \ref{sec:subgroup-support} treats multisets supported in a subgroup.
Section \ref{sec:S3} gives the counterexample in \(S_3\). Section
\ref{sec:infinite-family} extends the obstruction to the family
\(S_3\times K\) with \(3\nmid |K|\). The final section records several
concrete problems for further study.

\section{Quotient-realizable multisets}\label{sec:quotient-realizable}

Throughout the paper, \(G\) denotes a finite group with identity element \(1\).
All multisets are counted with multiplicity.

\begin{definition}\label{def:quotient-realizable}
Let \(A\) be a multiset of cardinality \(|G|\) whose elements belong to \(G\).
We say that \(A\) is \emph{quotient-realizable} in \(G\) if there exist
bijections
\[
        b,c:\{1,\ldots,|G|\}\to G
\]
such that
\[
        A=\{b(i)c(i)^{-1}:1\le i\le |G|\}
\]
as multisets.
\end{definition}

It is convenient to replace the pair of bijections by a single permutation of
\(G\).

\begin{lemma}\label{lem:single-permutation}
Let \(A\) be a multiset of cardinality \(|G|\) in \(G\). Then \(A\) is
quotient-realizable in \(G\) if and only if there exists a permutation
\(\varphi\in \Sym(G)\) such that
\[
        A=\{\varphi(x)x^{-1}:x\in G\}.
\]
\end{lemma}

\begin{proof}
Suppose first that
\[
        A=\{b(i)c(i)^{-1}:1\le i\le |G|\}
\]
for bijections \(b,c:\{1,\ldots,|G|\}\to G\). For each \(x\in G\), there is a
unique index \(i\) such that \(c(i)=x\). Define
\[
        \varphi(x)=b(i).
\]
Since both \(b\) and \(c\) are bijections, \(\varphi\) is a permutation of
\(G\). Moreover,
\[
        \varphi(x)x^{-1}=b(i)c(i)^{-1}.
\]
Therefore
\[
        A=\{\varphi(x)x^{-1}:x\in G\}.
\]

Conversely, suppose such a permutation \(\varphi\) is given. Enumerate
\[
        G=\{x_1,\ldots,x_{|G|}\}.
\]
Define
\[
        c(i)=x_i,\qquad b(i)=\varphi(x_i).
\]
Then \(b\) and \(c\) are bijections, and
\[
        b(i)c(i)^{-1}=\varphi(x_i)x_i^{-1}.
\]
Thus \(A\) is quotient-realizable.
\end{proof}

\begin{proposition}\label{prop:abelianization-obstruction}
Let \(A\) be a quotient-realizable multiset in \(G\), and let
\[
        \pi:G\to G_{\mathrm{ab}}=G/[G,G]
\]
be the canonical projection. Then
\[
        \prod_{a\in A}^{\mathrm{mult}} \pi(a)=1
        \qquad\text{in }G_{\mathrm{ab}}.
\]
Here the product is taken with multiplicity and is independent of the chosen
enumeration of \(A\), since \(G_{\mathrm{ab}}\) is abelian.
\end{proposition}

\begin{proof}
Choose bijections
\[
        b,c:\{1,\ldots,|G|\}\to G
\]
such that
\[
        A=\{b(i)c(i)^{-1}:1\le i\le |G|\}
\]
as multisets. Since \(G_{\mathrm{ab}}\) is abelian, we have
\[
        \prod_{a\in A}^{\mathrm{mult}} \pi(a)
        =
        \prod_{i=1}^{|G|} \pi\bigl(b(i)c(i)^{-1}\bigr)
        =
        \left(\prod_{i=1}^{|G|}\pi(b(i))\right)
        \left(\prod_{i=1}^{|G|}\pi(c(i))\right)^{-1}.
\]
Because \(b\) and \(c\) are both bijections onto \(G\), the multisets
\[
        \{\pi(b(i)):1\le i\le |G|\}
        \quad\text{and}\quad
        \{\pi(c(i)):1\le i\le |G|\}
\]
are equal. Hence the two products are equal in \(G_{\mathrm{ab}}\), and
therefore
\[
        \prod_{a\in A}^{\mathrm{mult}} \pi(a)=1.
\]
\end{proof}

\begin{remark}
If \(G\) is abelian, Proposition \ref{prop:abelianization-obstruction} is
exactly Hall's zero-sum condition written multiplicatively. The results below
show that, for nonabelian groups, this condition is not sufficient.
\end{remark}

\section{The cycle-tiling criterion}\label{sec:cycle-tiling}

We now prove the main structural theorem.

\begin{definition}\label{def:simple-word}
Let
\[
        w=(g_1,\ldots,g_\ell)
\]
be a word in \(G\). Define its left partial products by
\[
        p_0=1,\qquad
        p_j=g_jg_{j-1}\cdots g_1
        \quad (1\le j\le \ell).
\]
We say that \(w\) is a \emph{simple product-one word} if
\[
        p_\ell=1
\]
and the elements
\[
        p_0,p_1,\ldots,p_{\ell-1}
\]
are pairwise distinct. In this case we define
\[
        P(w)=\{p_0,p_1,\ldots,p_{\ell-1}\}.
\]
\end{definition}

\begin{remark}
A word of length \(1\) is simple product-one exactly when it is \((1)\). This
corresponds to a loop in the cycle decomposition: the vertex is fixed by the
corresponding permutation.
\end{remark}

\begin{definition}\label{def:cycle-tiling}
Let \(A\) be a multiset of cardinality \(|G|\) in \(G\). If
\[
        w=(g_1,\ldots,g_\ell)
\]
is a word in \(G\), we write
\[
        [w]=\{g_1,\ldots,g_\ell\}_{\rm mult}
\]
for the underlying multiset of its letters.

A \emph{cycle-tiling decomposition} of \(A\) is a collection of simple
product-one words
\[
        w_1,\ldots,w_r
\]
such that
\[
        A=[w_1]\sqcup\cdots\sqcup [w_r]
\]
as multisets, and such that there exist elements \(x_1,\ldots,x_r\in G\)
satisfying
\[
        G=\bigsqcup_{j=1}^r P(w_j)x_j.
\]
Here \(\bigsqcup\) denotes disjoint union, and
\[
        P(w_j)x_j=\{px_j:p\in P(w_j)\}.
\]
\end{definition}

\begin{theorem}\label{thm:cycle-tiling}
Let \(G\) be a finite group, and let \(A\) be a multiset of cardinality
\(|G|\) in \(G\). Then \(A\) is quotient-realizable in \(G\) if and only if
\(A\) admits a cycle-tiling decomposition.
\end{theorem}

\begin{proof}
Suppose first that \(A\) is quotient-realizable in \(G\). By
Lemma~\ref{lem:single-permutation}, there exists a permutation
\(\varphi\in\Sym(G)\) such that
\[
        A=\{\varphi(x)x^{-1}:x\in G\}
\]
as multisets.

Decompose \(\varphi\) into disjoint directed cycles. Consider one such cycle,
written as
\[
        x_0\longmapsto x_1\longmapsto \cdots
        \longmapsto x_{\ell-1}\longmapsto x_0 .
\]
The edge labels on this cycle are
\[
        g_j=x_jx_{j-1}^{-1}\qquad (1\le j\le \ell-1),
\]
and
\[
        g_\ell=x_0x_{\ell-1}^{-1}.
\]
Thus \(g_1,\ldots,g_\ell\) are precisely the elements
\(\varphi(x)x^{-1}\) arising from this directed cycle, listed in cyclic order.

Let
\[
        w=(g_1,\ldots,g_\ell).
\]
We show that \(w\) is a simple product-one word. Define
\[
        p_0=1,\qquad
        p_j=g_jg_{j-1}\cdots g_1
        \quad (1\le j\le \ell).
\]
For \(1\le j\le \ell-1\), telescoping gives
\[
        p_j
        =
        (x_jx_{j-1}^{-1})(x_{j-1}x_{j-2}^{-1})
        \cdots (x_1x_0^{-1})
        =
        x_jx_0^{-1}.
\]
Moreover,
\[
        p_\ell
        =
        g_\ell p_{\ell-1}
        =
        (x_0x_{\ell-1}^{-1})(x_{\ell-1}x_0^{-1})
        =
        1.
\]
Since the vertices
\[
        x_0,x_1,\ldots,x_{\ell-1}
\]
are pairwise distinct, the elements
\[
        p_0,p_1,\ldots,p_{\ell-1}
\]
are pairwise distinct. Hence \(w\) is a simple product-one word.

Furthermore,
\[
        P(w)x_0
        =
        \{p_0x_0,p_1x_0,\ldots,p_{\ell-1}x_0\}
        =
        \{x_0,x_1,\ldots,x_{\ell-1}\}.
\]
Thus the right translate \(P(w)x_0\) is exactly the vertex set of this cycle.

Applying the same construction to every cycle of \(\varphi\), we obtain simple
product-one words
\[
        w_1,\ldots,w_r.
\]
For each \(j\), let \(x_j\) denote the initial vertex chosen for the cycle
corresponding to \(w_j\). Since the directed cycles of \(\varphi\) are
disjoint and cover \(G\), the translated partial-product sets are pairwise
disjoint and satisfy
\[
        G=\bigsqcup_{j=1}^r P(w_j)x_j.
\]
Also, the edge labels of all the cycles are exactly the elements of the
multiset
\[
        \{\varphi(x)x^{-1}:x\in G\}=A.
\]
Therefore the underlying multisets of the words satisfy
\[
        A=[w_1]\sqcup\cdots\sqcup [w_r].
\]
Hence \(A\) admits a cycle-tiling decomposition.

Conversely, suppose that \(A\) admits a cycle-tiling decomposition. Thus there
exist simple product-one words
\[
        w_1,\ldots,w_r,
\]
where
\[
        w_j=(g_{j,1},\ldots,g_{j,\ell_j}),
\]
such that
\[
        A=[w_1]\sqcup\cdots\sqcup [w_r]
\]
as multisets, and there exist elements \(x_1,\ldots,x_r\in G\) such that
\[
        G=\bigsqcup_{j=1}^r P(w_j)x_j.
\]

For each \(j\), define the partial products of \(w_j\) by
\[
        p_{j,0}=1,\qquad
        p_{j,k}=g_{j,k}g_{j,k-1}\cdots g_{j,1}
        \quad (1\le k\le \ell_j).
\]
Since \(w_j\) is a simple product-one word, we have
\[
        p_{j,\ell_j}=1,
\]
and the elements
\[
        p_{j,0},p_{j,1},\ldots,p_{j,\ell_j-1}
\]
are pairwise distinct.

On the translated set \(P(w_j)x_j\), define a directed cycle by
\[
        p_{j,k-1}x_j \longmapsto p_{j,k}x_j
        \qquad (1\le k\le \ell_j),
\]
where the final edge is interpreted using
\[
        p_{j,\ell_j}x_j=p_{j,0}x_j=x_j.
\]
Because the sets \(P(w_j)x_j\) are pairwise disjoint and their union is \(G\),
these directed cycles together define a permutation
\[
        \varphi\in\Sym(G).
\]

It remains to verify that the labels of this permutation are exactly the
letters of the words \(w_1,\ldots,w_r\). Consider the edge
\[
        p_{j,k-1}x_j \longmapsto p_{j,k}x_j.
\]
Its label is
\[
        (p_{j,k}x_j)(p_{j,k-1}x_j)^{-1}.
\]
Since
\[
        (p_{j,k-1}x_j)^{-1}=x_j^{-1}p_{j,k-1}^{-1},
\]
we get
\[
        (p_{j,k}x_j)(p_{j,k-1}x_j)^{-1}
        =
        p_{j,k}x_jx_j^{-1}p_{j,k-1}^{-1}
        =
        p_{j,k}p_{j,k-1}^{-1}.
\]
But
\[
        p_{j,k}=g_{j,k}p_{j,k-1},
\]
and therefore
\[
        p_{j,k}p_{j,k-1}^{-1}=g_{j,k}.
\]
Thus the labels on the cycle \(P(w_j)x_j\) are precisely
\[
        g_{j,1},\ldots,g_{j,\ell_j}.
\]
Consequently,
\[
        \{\varphi(x)x^{-1}:x\in G\}
        =
        [w_1]\sqcup\cdots\sqcup [w_r]
        =
        A
\]
as multisets. By Lemma~\ref{lem:single-permutation}, \(A\) is
quotient-realizable in \(G\).
\end{proof}

\begin{remark}
Theorem \ref{thm:cycle-tiling} gives an exact structural criterion for the
nonabelian problem. Unlike Hall's theorem in the abelian case, the criterion is
not merely numerical: it records how product-one words tile the group through
their partial-product sets. Together with Hall's theorem, it implies that for
finite abelian groups the cycle-tiling condition is equivalent to the zero-sum
condition. In nonabelian groups, the partial-product sets and their right
translates carry additional information which is not detected by the
abelianization.
\end{remark}

\section{Multisets supported in a subgroup}\label{sec:subgroup-support}

The cycle-tiling criterion becomes especially useful when all labels lie in a
subgroup. In that case the arrows cannot move from one right coset of the
subgroup to another. This observation turns the general problem into a
collection of smaller problems inside the subgroup.

\begin{theorem}\label{thm:subgroup-support}
Let \(H\le G\), and let \(A\) be a multiset of cardinality \(|G|\) supported in
\(H\). Then \(A\) is quotient-realizable in \(G\) if and only if \(A\) can be
partitioned into \([G:H]\) submultisets
\[
        A_1,\ldots,A_{[G:H]},
        \qquad |A_j|=|H|,
\]
such that each \(A_j\) is quotient-realizable in \(H\).
\end{theorem}

\begin{proof}
Suppose first that \(A\) is quotient-realizable in \(G\). Choose
\(\varphi\in \Sym(G)\) such that
\[
        A=\{\varphi(x)x^{-1}:x\in G\}.
\]
Since \(A\) is supported in \(H\), for every \(x\in G\) we have
\[
        \varphi(x)x^{-1}\in H.
\]
Hence
\[
        \varphi(x)\in Hx.
\]
Thus \(\varphi\) maps every element \(x\) into the same right coset \(Hx\).
Therefore each right coset of \(H\) is invariant under \(\varphi\).

Let
\[
        G=Hx_1\sqcup\cdots\sqcup Hx_r,
        \qquad r=[G:H],
\]
be the right-coset decomposition. For each \(j\), define \(A_j\) to be the
multiset
\[
        A_j=\{\varphi(y)y^{-1}:y\in Hx_j\}.
\]
Then \(|A_j|=|H|\), and
\[
        A=A_1\sqcup\cdots\sqcup A_r
\]
as multisets.

It remains to prove that \(A_j\) is quotient-realizable in \(H\). Since
\(\varphi\) preserves \(Hx_j\), for every \(h\in H\) there exists a unique
element \(\psi_j(h)\in H\) such that
\[
        \varphi(hx_j)=\psi_j(h)x_j.
\]
The map \(\psi_j:H\to H\) is a permutation because \(\varphi\) restricts to a
permutation of the coset \(Hx_j\). Moreover,
\[
        \varphi(hx_j)(hx_j)^{-1}
        =
        \psi_j(h)x_jx_j^{-1}h^{-1}
        =
        \psi_j(h)h^{-1}.
\]
Thus
\[
        A_j=\{\psi_j(h)h^{-1}:h\in H\}.
\]
So \(A_j\) is quotient-realizable in \(H\).

Conversely, suppose
\[
        A=A_1\sqcup\cdots\sqcup A_r,
        \qquad r=[G:H],
\]
where \(|A_j|=|H|\) and each \(A_j\) is quotient-realizable in \(H\). Choose
right coset representatives
\[
        x_1,\ldots,x_r
\]
for \(H\) in \(G\). For each \(j\), choose a permutation
\(\psi_j\in \Sym(H)\) such that
\[
        A_j=\{\psi_j(h)h^{-1}:h\in H\}.
\]
Define \(\varphi:G\to G\) by
\[
        \varphi(hx_j)=\psi_j(h)x_j
        \qquad (h\in H,\ 1\le j\le r).
\]
This is well-defined because every element of \(G\) has a unique expression
\(hx_j\) with \(h\in H\). Since each \(\psi_j\) is a permutation of \(H\),
\(\varphi\) is a permutation of \(G\). Finally,
\[
        \varphi(hx_j)(hx_j)^{-1}
        =
        \psi_j(h)x_jx_j^{-1}h^{-1}
        =
        \psi_j(h)h^{-1}.
\]
Therefore the multiset of quotients \(\{\varphi(x)x^{-1}:x\in G\}\) is
precisely \(A\). Hence \(A\) is quotient-realizable in \(G\).
\end{proof}

\begin{corollary}\label{cor:abelian-subgroup}
Let \(H\le G\) be abelian, and let \(A\) be a multiset of cardinality \(|G|\)
supported in \(H\). Then \(A\) is quotient-realizable in \(G\) if and only if
\(A\) can be partitioned into \([G:H]\) submultisets
\[
        A_1,\ldots,A_{[G:H]},
        \qquad |A_j|=|H|,
\]
such that
\[
        \prod_{a\in A_j}a=1
\]
for every \(j\).
\end{corollary}

\begin{proof}
By Theorem \ref{thm:subgroup-support}, quotient-realizability of \(A\) in
\(G\) is equivalent to a partition of \(A\) into \([G:H]\) blocks of size
\(|H|\), each of which is quotient-realizable in \(H\). Since \(H\) is
abelian, Hall's theorem applies: a multiset \(B\) of cardinality \(|H|\) in
\(H\) is quotient-realizable in \(H\) if and only if
\[
        \prod_{b\in B}b=1.
\]
Applying this criterion to each block gives the claim.
\end{proof}

\begin{corollary}\label{cor:cyclic-subgroup}
Let \(g\in G\) have order \(d\), and set \(H=\langle g\rangle\). Let \(A\) be
a multiset of cardinality \(|G|\) supported in \(H\). Write the elements of
\(A\) as powers of \(g\). Then \(A\) is quotient-realizable in \(G\) if and
only if \(A\) can be partitioned into \([G:H]\) blocks
\[
        A_1,\ldots,A_{[G:H]},
        \qquad |A_j|=d,
\]
such that, whenever
\[
        A_j=\{g^{e_{j,1}},\ldots,g^{e_{j,d}}\},
\]
one has
\[
        e_{j,1}+\cdots+e_{j,d}\equiv 0\pmod d.
\]
\end{corollary}

\begin{proof}
This is Corollary \ref{cor:abelian-subgroup} applied to the cyclic subgroup
\(H=\langle g\rangle\). For a block \(A_j\subseteq H\),
\[
        \prod_{a\in A_j}a=1
\]
if and only if
\[
        g^{e_{j,1}+\cdots+e_{j,d}}=1,
\]
which is equivalent to
\[
        e_{j,1}+\cdots+e_{j,d}\equiv 0\pmod d.
\]
\end{proof}

\section{The first obstruction: \texorpdfstring{$S_3$}{S3}}\label{sec:S3}

We now give a concrete example showing that the abelianization obstruction is
not sufficient.

Let
\[
        G=S_3,\qquad s=(12),\qquad t=(23).
\]
We use the convention that permutations are composed from right to left. Thus
products act on the rightmost element first. With this convention,
\[
        s^2=t^2=1,\qquad (st)^3=(ts)^3=1.
\]
This convention is consistent with our left-partial-product convention
\[
        p_j=g_jg_{j-1}\cdots g_1.
\]
For example, for the alternating word
\[
        w=(s,t,s,t,s,t),
\]
the full product is
\[
        p_6=tststs=(ts)^3=1.
\]

Consider the multiset
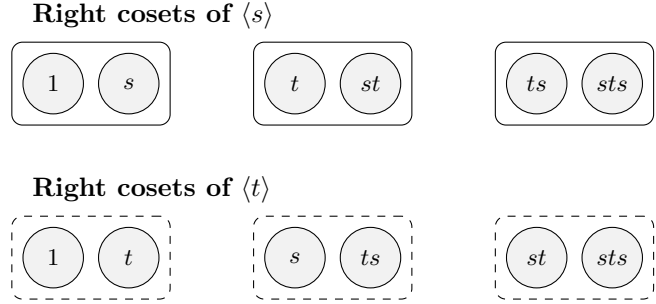
\[
        A=\{s,s,t,t,t,t\}.
\]

\begin{figure}[!htbp]
\centering
\begin{tikzpicture}[every node/.style={font=\small}]
  \node[anchor=west,font=\bfseries] at (-0.4,2.1) {Right cosets of $\langle s\rangle$};
  \node[circle,draw,fill=gray!10,minimum size=8mm] (a0) at (0,1.2) {$1$};
  \node[circle,draw,fill=gray!10,minimum size=8mm] (b0) at (1.0,1.2) {$s$};
  \node[circle,draw,fill=gray!10,minimum size=8mm] (a1) at (3.2,1.2) {$t$};
  \node[circle,draw,fill=gray!10,minimum size=8mm] (b1) at (4.2,1.2) {$st$};
  \node[circle,draw,fill=gray!10,minimum size=8mm] (a2) at (6.4,1.2) {$ts$};
  \node[circle,draw,fill=gray!10,minimum size=8mm] (b2) at (7.4,1.2) {$sts$};
  \begin{scope}[on background layer]
    \node[draw=black,rounded corners,fit=(a0)(b0),inner sep=4pt] {};
    \node[draw=black,rounded corners,fit=(a1)(b1),inner sep=4pt] {};
    \node[draw=black,rounded corners,fit=(a2)(b2),inner sep=4pt] {};
  \end{scope}

  \node[anchor=west,font=\bfseries] at (-0.4,-0.2) {Right cosets of $\langle t\rangle$};
  \node[circle,draw,fill=gray!10,minimum size=8mm] (c0) at (0,-1.1) {$1$};
  \node[circle,draw,fill=gray!10,minimum size=8mm] (d0) at (1.0,-1.1) {$t$};
  \node[circle,draw,fill=gray!10,minimum size=8mm] (c1) at (3.2,-1.1) {$s$};
  \node[circle,draw,fill=gray!10,minimum size=8mm] (d1) at (4.2,-1.1) {$ts$};
  \node[circle,draw,fill=gray!10,minimum size=8mm] (c2) at (6.4,-1.1) {$st$};
  \node[circle,draw,fill=gray!10,minimum size=8mm] (d2) at (7.4,-1.1) {$sts$};
  \begin{scope}[on background layer]
    \node[draw=black,dashed,rounded corners,fit=(c0)(d0),inner sep=4pt] {};
    \node[draw=black,dashed,rounded corners,fit=(c1)(d1),inner sep=4pt] {};
    \node[draw=black,dashed,rounded corners,fit=(c2)(d2),inner sep=4pt] {};
  \end{scope}

  \node[draw,rounded corners,fill=gray!8,text width=0.88\linewidth,align=left] at (4.2,-2.75) {
  For $A=\{s,s,t,t,t,t\}$, any quotient realization would have to use one tile
  of type $\langle s\rangle x$ and two tiles of type $\langle t\rangle y$.
  Proposition~\ref{prop:S3-counterexample} shows that such a mixed coset tiling
  of $S_3$ is impossible.};
\end{tikzpicture}
\caption{The two right-coset partitions used in the obstruction in \(S_3\).
Solid boxes indicate right cosets of \(\langle s\rangle\), while dashed boxes
indicate right cosets of \(\langle t\rangle\). The figure illustrates why a
mixed tiling by one coset of \(\langle s\rangle\) and two cosets of
\(\langle t\rangle\) would force a coset of \(\langle s\rangle\) to coincide
with a coset of \(\langle t\rangle\), which is impossible.}
\label{fig:s3cosets}
\end{figure}

We first classify the simple product-one words in \(S_3\) using only the
letters \(s\) and \(t\).

\begin{lemma}\label{lem:S3-words}
Let
\[
        w=(g_1,\ldots,g_\ell)
\]
be a simple product-one word in \(S_3\), where each \(g_i\) belongs to
\(\{s,t\}\). Then exactly one of the following holds:
\begin{enumerate}[label=\textup{(\roman*)}]
    \item \(w=(s,s)\);
    \item \(w=(t,t)\);
    \item \(\ell=6\), and \(w\) is alternating in \(s\) and \(t\). In
    particular, \(w\) contains three occurrences of \(s\) and three
    occurrences of \(t\).
\end{enumerate}
\end{lemma}

\begin{proof}
Let
\[
        p_0=1,\qquad p_j=g_jg_{j-1}\cdots g_1.
\]
Since \(w\) is simple product-one, \(p_\ell=1\), and
\[
        p_0,p_1,\ldots,p_{\ell-1}
\]
are pairwise distinct.

Suppose first that \(g_i=g_{i+1}\) for some \(1\le i<\ell\). Since both \(s\)
and \(t\) have order \(2\), we have
\[
        p_{i+1}
        =
        g_{i+1}p_i
        =
        g_i(g_ip_{i-1})
        =
        p_{i-1}.
\]
If \(i+1\le \ell-1\), this contradicts the pairwise distinctness of
\(p_0,\ldots,p_{\ell-1}\). If \(i+1=\ell\), then \(p_\ell=p_{i-1}\). Since
\(p_\ell=1=p_0\), this gives \(p_{i-1}=p_0\), contradicting the pairwise
distinctness of \(p_0,\ldots,p_{\ell-1}\) unless \(i-1=0\). In the latter
case \(\ell=2\). Therefore the only possible words with two consecutive equal
letters are \((s,s)\) and \((t,t)\).

It remains to consider the case in which no two consecutive letters
\(g_i,g_{i+1}\), with \(1\le i<\ell\), are equal. Suppose that \(g_\ell=g_1\).
Since \(p_\ell=1\), we have
\[
        p_{\ell-1}=g_\ell^{-1}=g_\ell=g_1=p_1.
\]
If \(\ell>2\), this contradicts the pairwise distinctness of
\(p_0,\ldots,p_{\ell-1}\). Thus, outside the two length-two cases already
identified, we must have no consecutive equality even cyclically. Hence the
word is cyclically alternating in \(s\) and \(t\). In particular, \(\ell\) is
even.

If the word begins with \(s\), then its full product \(p_\ell\) is a power of
\(ts\); if it begins with \(t\), then its full product is a power of \(st\).
Both \(st\) and \(ts\) have order \(3\) in \(S_3\). Hence an alternating word
has product \(1\) precisely when its length is divisible by \(6\). Since a
simple product-one word has \(\ell\) distinct partial products
\[
        p_0,p_1,\ldots,p_{\ell-1}
\]
in the group \(S_3\), we must have \(\ell\le 6\). Therefore \(\ell=6\).
Such a word contains three occurrences of \(s\) and three occurrences of
\(t\).
\end{proof}

\begin{proposition}\label{prop:S3-counterexample}
The multiset
\[
        A=\{s,s,t,t,t,t\}
\]
is not quotient-realizable in \(S_3\).
\end{proposition}

\begin{proof}
Assume, for a contradiction, that \(A\) is quotient-realizable in \(S_3\).
By Theorem~\ref{thm:cycle-tiling}, \(A\) admits a cycle-tiling decomposition.

Since
\[
        A=\{s,s,t,t,t,t\},
\]
every word appearing in such a decomposition has all of its letters in
\(\{s,t\}\). By Lemma~\ref{lem:S3-words}, every simple product-one word with
letters in \(\{s,t\}\) is either
\[
        (s,s),
        \qquad
        (t,t),
\]
or an alternating word of length \(6\) containing three copies of \(s\) and
three copies of \(t\).

The multiset \(A\) contains only two copies of \(s\). Therefore no alternating
word of length \(6\) can appear in the decomposition, since such a word would
require three copies of \(s\). Hence the only possible decomposition of \(A\)
into simple product-one words is
\[
        (s,s),\qquad (t,t),\qquad (t,t),
\]
up to reordering of these three words.

The partial-product set of the word \((s,s)\) is
\[
        \{1,s\}=\langle s\rangle.
\]
Thus any tile arising from \((s,s)\) is a right coset of \(\langle s\rangle\).
Similarly, the partial-product set of \((t,t)\) is
\[
        \{1,t\}=\langle t\rangle,
\]
so any tile arising from \((t,t)\) is a right coset of \(\langle t\rangle\).

Therefore a cycle-tiling decomposition of \(A\) would give a disjoint
partition of \(S_3\) into one right coset of \(\langle s\rangle\) and two
right cosets of \(\langle t\rangle\).

We now show that no such partition exists. The subgroup \(\langle t\rangle\)
has three right cosets in \(S_3\). The two right cosets of \(\langle t\rangle\)
appearing in the proposed tiling must be distinct, because the tiling is
disjoint. Their union therefore has complement equal to the remaining right
coset of \(\langle t\rangle\). Hence, if \(S_3\) were the disjoint union of
one right coset of \(\langle s\rangle\) and two right cosets of
\(\langle t\rangle\), then that right coset of \(\langle s\rangle\) would have
to be equal to a right coset of \(\langle t\rangle\).

We use the elementary fact that two right cosets \(Hx\) and \(Ky\) of
subgroups \(H,K\le G\) can be equal only if \(H=K\). Indeed, if \(Hx=Ky\),
choose an element \(z\in Hx=Ky\). Then
\[
        Hx=Hz,\qquad Ky=Kz,
\]
and hence
\[
        Hz=Kz.
\]
Multiplying on the right by \(z^{-1}\) gives
\[
        H=K.
\]
Applying this fact with \(H=\langle s\rangle\) and \(K=\langle t\rangle\),
we would obtain
\[
        \langle s\rangle=\langle t\rangle.
\]
This is impossible, since \(s=(12)\) and \(t=(23)\) generate distinct
subgroups of order \(2\).

This contradiction shows that \(A\) is not quotient-realizable in \(S_3\).
\end{proof}

\begin{remark}
The obstruction in Proposition \ref{prop:S3-counterexample} is genuinely a tiling obstruction. Indeed, the multiset
\[
A=\{s,s,t,t,t,t\}
\]
can be partitioned into simple product-one words as
\[
A=[(s,s)]\sqcup[(t,t)]\sqcup[(t,t)].
\]
Thus the failure of quotient-realizability does not come from the absence of product-one pieces. Rather, the corresponding partial-product sets are
\[
P((s,s))=\{1,s\}=\langle s\rangle,
\qquad
P((t,t))=\{1,t\}=\langle t\rangle.
\]
Proposition \ref{prop:S3-counterexample} shows that these sets cannot be placed as disjoint right translates whose union is $S_3$. In this sense, quotient-realizability is stronger than both the abelianization condition and the existence of a product-one ordering: the product-one pieces must also fit together as a cycle-tiling of the group.
\end{remark}

\begin{proposition}\label{prop:S3-admissible}
The multiset
\[
        A=\{s,s,t,t,t,t\}
\]
satisfies the abelianization obstruction and admits an ordering whose product
is \(1\).
\end{proposition}

\begin{proof}
The abelianization of \(S_3\) is isomorphic to \(C_2\), and every
transposition maps to the nontrivial element of \(C_2\). Since \(A\) contains
six transpositions, the product of the images of its elements in the
abelianization is trivial. Thus \(A\) satisfies the abelianization
obstruction.

Moreover,
\[
        ss\,tt\,tt=1,
\]
because \(s^2=t^2=1\). Hence the elements of \(A\) can be ordered so that
their product is \(1\).
\end{proof}

\begin{corollary}\label{cor:S3-nonsufficient}
There exists a finite nonabelian group for which the abelianization condition
is not sufficient for quotient-realizability. Moreover, the existence of an
ordering of the multiset with product \(1\) is not sufficient.
\end{corollary}

\section{The obstruction persists in infinitely many groups}
\label{sec:infinite-family}

The counterexample in \(S_3\) is not isolated. In this section we use the
subgroup-support theorem to construct an infinite family of finite nonabelian
groups for which the abelianization obstruction is not sufficient.

We first record a sharper form of the \(S_3\) obstruction.

\begin{lemma}\label{lem:S3-counting}
Let
\[
        G=S_3,\qquad s=(12),\qquad t=(23).
\]
Let \(B\) be a quotient-realizable multiset of cardinality \(6\) supported in
\(\{s,t\}\). If \(B\) contains \(q\) copies of \(s\), then
\[
        q\in \{0,3,6\}.
\]
Conversely, for each \(q\in\{0,3,6\}\), there exists a quotient-realizable
multiset of cardinality \(6\) supported in \(\{s,t\}\) containing exactly
\(q\) copies of \(s\).
\end{lemma}

\begin{proof}
Suppose first that \(B\) is quotient-realizable. By Theorem
\ref{thm:cycle-tiling}, \(B\) admits a cycle-tiling decomposition into simple
product-one words. Since \(B\) is supported in \(\{s,t\}\), every word in this
decomposition has all of its letters in \(\{s,t\}\).

By Lemma \ref{lem:S3-words}, every such simple product-one word is either
\[
        (s,s),\qquad (t,t),
\]
or an alternating word of length \(6\), containing three copies of \(s\) and
three copies of \(t\).

If the decomposition contains an alternating word of length \(6\), then that
word uses all six letters of \(B\). Hence \(B\) contains exactly three copies
of \(s\).

Assume now that no alternating word occurs. Then the decomposition consists
only of words of type \((s,s)\) and \((t,t)\). Since \(|B|=6\), there are
three such words. If \(a\) of them are of type \((s,s)\), then \(B\) contains
\[
        q=2a
\]
copies of \(s\), where \(a\in\{0,1,2,3\}\).

The partial-product set of \((s,s)\) is
\[
        \{1,s\}=\langle s\rangle,
\]
so a tile arising from \((s,s)\) is a right coset of \(\langle s\rangle\).
Similarly, a tile arising from \((t,t)\) is a right coset of
\(\langle t\rangle\).

If \(a=1\), then \(S_3\) would be partitioned into one right coset of
\(\langle s\rangle\) and two right cosets of \(\langle t\rangle\). This is
impossible by the coset argument in the proof of
Proposition \ref{prop:S3-counterexample}. If \(a=2\), then \(S_3\) would be
partitioned into two right cosets of \(\langle s\rangle\) and one right coset
of \(\langle t\rangle\). Taking complements gives the same contradiction: a
right coset of \(\langle s\rangle\) would have to be equal to a right coset of
\(\langle t\rangle\), which is impossible because
\[
        \langle s\rangle\ne \langle t\rangle.
\]
Therefore \(a\ne 1,2\), and hence
\[
        q\in\{0,3,6\}.
\]

It remains to prove the converse. If \(q=0\), then
\[
        B=\{t,t,t,t,t,t\}.
\]
Partition \(S_3\) into its three right cosets of \(\langle t\rangle\), and on
each coset use a \(2\)-cycle with labels \((t,t)\). This gives a quotient
realization. The case \(q=6\) is identical, using the three right cosets of
\(\langle s\rangle\).

Finally, suppose \(q=3\). Consider the alternating word
\[
        w=(s,t,s,t,s,t).
\]
With our convention for left partial products,
\[
        p_0=1,\qquad p_j=g_jg_{j-1}\cdots g_1.
\]
The full product is
\[
        p_6=tststs=(ts)^3=1.
\]
The partial products
\[
        p_0,p_1,\ldots,p_5
\]
are distinct, because otherwise the word would close earlier and would give
an alternating product-one word of length strictly smaller than \(6\), which
is impossible since \(st\) and \(ts\) have order \(3\). Equivalently, one may
check directly that these six partial products are the six elements of
\(S_3\). Hence \(w\) is a simple product-one word with
\[
        P(w)=S_3.
\]
Thus \(w\) gives a one-cycle quotient realization of a multiset with three
copies of \(s\) and three copies of \(t\).
\end{proof}

\begin{theorem}\label{thm:infinite-family}
Let \(K\) be a finite group such that \(3\nmid |K|\), and set
\[
        G=S_3\times K.
\]
Let
\[
        s'=((12),1_K),\qquad t'=((23),1_K).
\]
Let \(A\) be the multiset in \(G\) consisting of \(2|K|\) copies of \(s'\) and
\(4|K|\) copies of \(t'\). Then \(A\) satisfies the abelianization obstruction
and admits an ordering whose product is \(1\), but \(A\) is not
quotient-realizable in \(G\).
\end{theorem}

\begin{proof}
Let
\[
        m=|K|.
\]
The multiset \(A\) has cardinality
\[
        2m+4m=6m=|S_3\times K|=|G|.
\]

First we verify the two product conditions. In the abelianization of
\(G=S_3\times K\), the elements \(s'\) and \(t'\) have the same image coming
from the nontrivial element of the abelianization of \(S_3\). This image has
order \(2\). Since \(A\) contains
\[
        2m+4m=6m
\]
such elements, an even number, the product of their images in
\(G_{\mathrm{ab}}\) is trivial. Hence \(A\) satisfies the abelianization
obstruction.

Also, \(A\) admits an ordering whose product is \(1\). Indeed, arrange the
\(2m\) copies of \(s'\) into \(m\) adjacent pairs and the \(4m\) copies of
\(t'\) into \(2m\) adjacent pairs. Since
\[
        (s')^2=(t')^2=1,
\]
the resulting ordered product is \(1\).

We now prove that \(A\) is not quotient-realizable. Let
\[
        H=S_3\times\{1_K\}.
\]
Then \(H\le G\), \(|H|=6\), and
\[
        [G:H]=m.
\]
The multiset \(A\) is supported in \(H\).

Assume, for a contradiction, that \(A\) is quotient-realizable in \(G\). By
Theorem~\ref{thm:subgroup-support}, \(A\) can be partitioned into \(m\)
submultisets
\[
        A=A_1\sqcup\cdots\sqcup A_m,
        \qquad |A_j|=6,
\]
such that each \(A_j\) is quotient-realizable in \(H\).

Identifying \(H\) with \(S_3\), each \(A_j\) is a quotient-realizable multiset
of cardinality \(6\) supported in \(\{s',t'\}\). Let \(q_j\) be the number of
copies of \(s'\) in \(A_j\). By Lemma~\ref{lem:S3-counting},
\[
        q_j\in\{0,3,6\}
        \qquad (1\le j\le m).
\]
Therefore
\[
        \sum_{j=1}^m q_j
\]
is divisible by \(3\). On the other hand, this sum is exactly the total number
of copies of \(s'\) in \(A\), namely
\[
        \sum_{j=1}^m q_j=2m.
\]
Since \(3\nmid m\), we have \(3\nmid 2m\), a contradiction. Hence \(A\) is not
quotient-realizable in \(G\).
\end{proof}

\begin{corollary}\label{cor:infinitely-many}
There are infinitely many finite nonabelian groups \(G\) admitting multisets
\(A\) of cardinality \(|G|\) such that:
\begin{enumerate}[label=\textup{(\roman*)}]
    \item \(A\) satisfies the abelianization obstruction;
    \item the elements of \(A\) can be ordered with product \(1\);
    \item \(A\) is not quotient-realizable in \(G\).
\end{enumerate}
\end{corollary}

\begin{proof}
Take
\[
        G=S_3\times K,
\]
where \(K\) ranges over any infinite family of finite groups whose orders are
not divisible by \(3\); for example, take \(K=C_m\) with \(3\nmid m\).
Theorem \ref{thm:infinite-family} applies.
\end{proof}

\section{Further problems}\label{sec:problems}

The examples above show that the abelianization condition does not control the
nonabelian problem. The remaining challenge is to understand how often the
cycle-tiling obstruction appears, and whether it can be described more
explicitly for familiar families of finite groups.

\begin{problem}
Classify all quotient-realizable multisets of cardinality \(6\) in \(S_3\).
\end{problem}

The present paper studies one particular family of multisets in \(S_3\). A
complete classification would clarify the full range of cycle-tiling
obstructions in the smallest nonabelian group.

\begin{problem}
Let \(D_8\) be the dihedral group of order \(8\). Find a multiset \(A\) of
cardinality \(8\) satisfying the abelianization condition but not
quotient-realizable, or prove that no such multiset exists.
\end{problem}

\begin{problem}
Do the same for the quaternion group \(Q_8\).
\end{problem}

The comparison between \(D_8\) and \(Q_8\) may help identify which structural
features of a nonabelian group are relevant to quotient-realizability.

\begin{problem}
Let \(G\) be a finite group and \(H\le G\) an abelian subgroup. Construct
multisets supported in \(H\) whose total product is trivial in the
abelianization of \(G\), but which cannot be partitioned into zero-sum blocks
of size \(|H|\).
\end{problem}

By Corollary \ref{cor:abelian-subgroup}, such multisets are not
quotient-realizable in \(G\). This problem connects the present paper with
zero-sum theory in finite abelian groups.

\begin{problem}
For which finite groups \(G\) is the abelianization obstruction sufficient for
quotient-realizability?
\end{problem}

Hall's theorem says that every finite abelian group has this property.
Theorem \ref{thm:infinite-family} and Corollary \ref{cor:infinitely-many}
show that infinitely many nonabelian groups do not.

\begin{problem}
Does every finite nonabelian group admit a multiset \(A\) of cardinality
\(|G|\) satisfying the abelianization obstruction but not quotient-realizable?
\end{problem}

A positive answer would characterize finite abelian groups as precisely those
finite groups for which the abelianization obstruction is the only obstruction.

\begin{problem}
Turn the cycle-tiling criterion into explicit tests for familiar families of
groups, such as dihedral groups, quaternion groups, or nilpotent groups of
class two.
\end{problem}

The criterion in Theorem \ref{thm:cycle-tiling} is exact, but more explicit
tests would be useful in applications and examples.

\end{document}